\documentclass[11pt,reqno]{amsart}

\usepackage{amsmath}
\usepackage{graphicx}
\usepackage{hyperref}
\usepackage{mathrsfs}
\usepackage{xcolor}
\usepackage{bm,bbm}
\usepackage{amsfonts}
\usepackage{amssymb}
\usepackage{csquotes}
\usepackage{stmaryrd}
\usepackage[normalem]{ulem}
\usepackage{cancel, soul}

\newcommand{\sredm}[1]{\ifmmode\text{\xout{\ensuremath{\displaystyle \textcolor{red}{#1}}}}\else\sout{\textcolor{red}{#1}}\fi}

\usepackage{enumerate}

\usepackage{soul}
\usepackage[margin=1.1in]{geometry}
\numberwithin{equation}{section}

\newtheorem{theorem}{Theorem}[section]
\newtheorem{lemma}{Lemma}[section]
\newtheorem{assumption}{Assumption}[section]
\newtheorem{proposition}{Proposition}[section]

\newtheorem{example}{Example}[section]
\newtheorem{definition}{Definition}[section]
\newtheorem{remark}{Remark}[section]

\newcommand{\noi}{\noindent}

\newcommand{\E}{\mathbb{E}}
\newcommand{\R}{\mathbb{R}}

\newcommand{\N}{\mathbb{N}}

\newcommand{\la}{\lambda}
\newcommand{\sig}{\sigma}

\newcommand{\eps}{\varepsilon}

\newcommand{\ph}{\varphi}

\newcommand{\al}{\alpha}
\newcommand{\gam}{\gamma}

\newcommand{\del}{\delta}

\newcommand{\Q}{{\mathbb Q}}

\newcommand{\PP}{{\mathbb P}}

\newcommand{\calC}{{\mathcal C}}
\newcommand{\calD}{{\mathcal D}}

\newcommand{\calF}{{\mathcal F}}

\newcommand{\calL}{{\mathcal L}}

\newcommand{\calQ}{{\mathcal Q}}

\newcommand{\calZ}{{\mathcal Z}}

\newcommand{\skp}{\vspace{\baselineskip}}

\newcommand\iy{\infty}

\newcommand{\Leps}{ \la^\eps}

\DeclareMathOperator*{\argmin}{arg\,min}

\newcommand{\one}{\mathbbm{1}}
\allowdisplaybreaks

\title{Optimal ergodic harvesting under ambiguity}
\author[A. Cohen]{Asaf Cohen }
\address{Department of Mathematics\\
University of Michigan\\
Ann Arbor, MI 48109\\
United States
}
\email{shloshim@gmail.com }

\author[A. Hening]{Alexandru Hening }
\address{Department of Mathematics\\
Tufts University\\
Bromfield-Pearson Hall\\
503 Boston Avenue\\
Medford, MA 02155\\
United States
}
\email{alexandru.hening@tufts.edu}

\author[C. Sun]{Chuhao Sun}
\address{Department of Mathematics\\
University of Michigan\\
Ann Arbor, MI 48109\\
United States
}
\email{chuhaos@umich.edu}

\thanks{A. Cohen acknowledges the financial support of Research supported by the National Science Foundation (DMS-2006305). A. Hening was partially supported via the National Science Foundation grant DMS-1853463.}

\date{\today}

\begin{document}

\maketitle

\begin{abstract}
We consider an ergodic harvesting problem with model ambiguity that arises from biology. To account for the ambiguity, the problem is constructed as a stochastic game with two players: the decision-maker (DM) chooses the `best' harvesting policy and an adverse player chooses the `worst' probability measure. The main result is establishing an optimal strategy (also referred to as a control) of the DM and showing that it is a threshold policy. The optimal threshold and the optimal payoff are obtained by solving a free-boundary problem emerging from the Hamilton--Jacobi--Bellman (HJB) equation. As part of the proof, we fix a gap that appeared in the HJB analysis of  [Alvarez and Hening, {\em Stochastic Process. Appl.}, 2019, \cite{alv-hen2019
}], a paper that analyzed the risk-neutral version of the ergodic harvesting problem. Finally, we study the dependence of the optimal threshold and the optimal payoff on the ambiguity parameter and show that if the ambiguity goes to 0, the problem converges, to the risk-neutral problem.

\skp

\skp

\noi{\bf AMS Classification:} Primary:
93E20,  
91A15,  
92D25;  
secondary:	49J15,  
60J70, 
35R35.  

\noi{\bf Keywords:} Ergodic control, singular control, model uncertainty, stochastic harvesting, optimal harvesting, stochastic games.
\end{abstract}

%
%
%
%
%
%
%

\section{Introduction}\label{sec:1}

%
%
%
%
%

\subsection{The stochastic model and the main results}\label{sec:11}
We consider the ergodic harvesting problem of a population that lives in a random environment, when there is uncertainty about the underlying model. For this, we assume that there is a {\it reference filtered probability space} $(\Omega,\calF, \{\calF_t\}_t,\PP)$ supporting a Wiener process $(W_t)_{t\in\R_+}$, such that the dynamics of the harvested population satisfies
\begin{align}\notag
X^Z_t=x+\int_0^tX^Z_s\mu(X^Z_s)ds+\int_0^t\sigma(X^Z_s) dW_s- Z_t,\qquad t\in\R_+,
\end{align}
where $(Z_t)_{t\in\R_+}$ is a singular control (nondecreasing and nonnegative) under which $  X^Z_t\ge 0$.
To account for the uncertainty, the decision maker (DM) considers a large set of measures, which are equivalent to $\PP$.
The DM then incorporates these measures into the payoff she aims to maximize as follows
\begin{align}\notag
\inf_{\Q}\E^{\Q}\left[\liminf_{T\to\iy} \frac{1}{T}\left\{\E^{ \Q}[Z_T]+\frac{1}{\eps}D^{\text{KL}}_T( \Q\|\PP)\right\}\right].
\end{align}
The infimum, which represents an {\it adverse player}, is taken over the set of the equivalent measures under consideration; $\eps>0$ is a parameter that measures the level of ambiguity the DM is facing; and the Kullback--Leibler divergence $D^{\text{KL}}_T(\Q\|\PP)$ measures how much the measure $\Q$ deviates from the reference measure $\PP$.

\vspace{10pt}\noi
Using differential equations techniques and probabilistic methods we show that for any level of ambiguity $\eps$, there is an optimal control for the DM. This control is of a threshold form, in the sense that there is a level $\beta^\eps>0$, such that the control uses minimal amount of harvesting in order to maintain the dynamics in the interval $[0,\beta^\eps]$. 
These properties are summarized in the main theorem of the paper, Theorem \ref{thm:main}. The proof relies on finding lower and upper bounds for the optimal payoff. The second main contribution of the paper is fixing a gap that appears in the upper bound part considered in 
another related ergodic harvesting model (without ambiguity), \cite{alv-hen2019}. Our upper bound result, Proposition \ref{prop:upper}, is valid with or without ambiguity. Furthermore, we prove that the optimal control's threshold-levels and the optimal payoffs are continuous and decreasing with respect to the ambiguity parameter and find their limiting behavior as the ambiguity paprmeter go to the extremes $\eps\to\iy, 0+$. Finally, when $\eps\to0+$ we obtain convergence to the risk-neutral problem studied by \cite{alv-hen2019}.

\vspace{10pt}\noi
This is one of the first models that incorporates ambiguity in ergodic singular control problems as well as in harvesting models. In the setting of ergodic control there are some results for controls that are not singular \cite{BCP17}. The ambiguity models that appear in the harvesting literature are extremely simple, and mostly look at linear stochastic differential equations models \cite{VX10}.

\subsection{Review of the literature.}\label{sec:12}

In (stochastic) ergodic control problems the goal of the DM is to optimize a time-averaged criterion over an infinite horizon. This theory was first developed for discrete-time and discrete-space Markov chains, see the survey \cite{ara1993}. The ideas and intuition carried out to continuous-time problems, see \cite{ABbook}, where due to the stability of the solution, it became popular in the analysis of stochastic networks; see e.g., \cite{kus2001} and the references therein. More recently, Alvarez and Hening \cite{alv-hen2019} took advantage of the stability properties of the state dynamics to study sustainable harvesting. We discuss this model as well as other biological models in the sequel.

\vspace{10pt}\noi
Singular control problems have been studied in various fields such as biology, queueing systems, mathematical finance, manufacturing systems, etc. 
The Hamilton--Jacobi--Bellman (HJB) equations associated with these models are often reduced to free-boundary condition problems with Neumann boundary conditions. 
Menaldi et.~al \cite{men-rob-tak1992} characterize the value of a singular control problem with ergodic cost (and constant diffusion coefficient) via the HJB, in case the latter has a smooth solution. The smoothness of the HJB is not obvious in general, in which case, viscosity solutions are considered, see e.g., \cite{Atar-Budh-Will-07}.
Buhiraja and Ross used probabilistic methods and time-transformations techniques to establish the existence of an optimal singular control in \cite{BR2006}. Recently, Cohen \cite{CoTime} showed that the time-transformations are embedded within the weak-M1 topology.

\vspace{10pt}\noi
One of the fundamental problems of conservation biology is finding the optimal ways of harvesting species which are influenced by stochastic environmental fluctuations. If one overharvests, this can lead to extinctions while if one underharvests this leads to an economic loss. There has been significant work on harvesting problems when the payoff function involves the discounted gain
\[
\E \int_0^\infty e^{-\varrho s} dZ_s,
\]
where $\varrho>0$ is the discount rate. Multiple studies have shown that the optimal control is of \textit{threshold} or \textit{bang-bang} type. These types of singular stochastic control problems have been investigated in \cite{LO97, AS98, A00, A01} in the single species case, in \cite{SSZ11} when there is switching and in \cite{LO01, ALO16, HT19, HT20} in the multispecies setting.
In \cite{HNUK19} the authors considered the framework of singular ergodic control for the first time for harvesting problems. This is a very natural setting because it implies that it is never optimal to harvest a species to extinction, as it would yield a zero asymptotic yield. This is therefore a realistic setting, if one cares about the preservation of natural species. One other problem with discounted models is that it is very hard to estimate the discount rate $\varrho$; the ergodic framework needs fewer parameters and is superior in this respect. One limitation of \cite{HNUK19} is that the harvesting rate is assumed to be bounded, i.e., the controls are not singular. This limitation was removed in \cite{alv-hen2019}.
However, as we explain in Remark
 \ref{rem:AH} 
below, there is a gap that is present in the proof of \cite{alv-hen2019}.

\vspace{10pt}\noi
As mentioned earlier, the ergodic harvesting problem without ambiguity was studied in \cite{alv-hen2019}. It was also studied in \cite{lia-zer2020}, where a running non-singular payoff is incorporated.
In these models it is assumed that the DM is certain about the evolution of the system, which, moreover does not change in time.  Such an assumption is not realistic, and we consider a robust analysis. When one is interested in the conservation and harvesting of a species there are certain obstacles which have to be overcome. One is the complexity of the biology, which has to be simplified in order to model the dynamics mathematically. A second difficulty is due to a number of uncertainties: the structure of environmental fluctuations and the fact that one never has a complete knowledge of the various population sizes. These uncertainties make it difficult to associate probabilities with certain events involving the population. This is related to an economic framework due to Knight, where there is incomplete or insufficient information in order to assign probabilities to events. We bypass these problems by adding the ambiguity, also called \textit{Knightian uncertainty}, to our model. The introduction of the ambiguity makes it possible to associate probabilities to events by looking at a set of multiple possible measures that control the population dynamics. Our robust framework will be helpful when one is interested in the conservation of a species because it allows us to explore the least favorable outcomes by looking at the infimum over all the possible measures (or \textit{priors}) of the payoff.
For further research that involves Knightian uncertainty we refer to  \cite{maenhout2004robust, Hansen2006, han-sar, bay-zha, neu-nut2018} and in the context of queueing systems to \cite{Shanti,blanchet2014robust,MR3544795, Cohen2019a, Cohen2019b, Cohen2020}. 

\subsection{Challenges and proof techniques}\label{sec:13}
The structure of the Kullback--Leibler divergence leads to a linear-quadratic (standard) control problem from the side of the adverse player; that is, using Girsanov's theorem, the measure $\Q$ can be replaced by a process $(\psi_t)_{t\in\R_+}$ and the divergence-penalization term is replaced by an integral over the square of $\psi_t$. This representation makes it possible to describe the unharvested population process by the nonlinear operator
\begin{align*}
\calL^\eps f(x)&:=\inf_{p\in\R}\Big\{\frac{1}{2}\sigma^2(x)f''(x)+(x\mu(x)+\sigma(x)p)f'(x)+\frac{1}{2\eps}p^2
\Big\}
\\
&\;=
\frac{1}{2}\sigma^2(x)f''(x)+x\mu(x)f'(x)-\frac{\eps}2\sigma^2(x)(f'(x))^2.
\end{align*}
The HJB associated with this model is given by \begin{align}\label{HJB1}
\max\left\{\calL^\eps v(x)-\ell, v'(x)- 1\right\}=0,\qquad x\in(0,\iy),
\end{align}
where, upon sufficient smoothness of the solution, $\ell\ge 0$ is the optimal function, and $f$ is referred to as the {\it potential function}. Our first challenge is to show that this equation admits a $\calC^2$ solution. In Proposition \ref{prop:HJB} we establish a stronger result and show that for any $\eps\ge 0$, there is $\beta^\eps>0$ such that $\max\left\{\calL^\eps v(x)-\ell, v'(x)- 1\right\}=\calL^\eps v(x)-\ell=0$ on $[0,\beta^\eps]$ and $\max\left\{\calL^\eps v(x)-\ell, v'(x)- 1\right\}= v'(x)- 1=0$ on $(\beta^\eps,\iy)$. This is translated to the stochastic model as a threshold-control with level $\beta^\eps$. The main difficulties here stem from the nonlinear structure of the operator $\calL^\eps$. To tackle these issues we use the {\it shooting method}, a tool for solving boundary value problems using initial value problems (see \cite[Section 7.3]{Stoer1980}).
We take it one step forward in our free-boundary setup. The analysis requires a sequence of preliminary results (Lemmas \ref{lemma:deriv}--\ref{lemma:unibou}) that evolve around an ODE that is derived from the HJB and our educated guess that the optimal policy is of threshold-type.

\vspace{10pt}\noi
Given a smooth solution $v^\eps$ to \eqref{HJB1}, with associated parameters $\beta^\eps, \ell^\eps$, we show in Proposition \ref{prop:lower} that the threshold-control with level $\beta^\eps$ attains the payoff $\ell^\eps$. This establishes a lower bound. The next step is to show that by using any other control, the DM cannot attain more than $\ell^\eps$. This in turn shows that $\ell^\eps$ is an upper bound for the optimal payoff. We accomplish this in Proposition \ref{prop:upper}. In the next few paragraphs, we detail the difficulties in establishing this bound and the solution we propose.

\vspace{10pt}\noi
In the proof of the upper bound, we consider an arbitrary admissible control $Z$ and fix the candidate for the optimal control for the adverse player. Then, applying It\^o's lemma to $v^\eps(X^Z_t)$ and using the properties that $(v^\eps)'(x)\ge 1$ and $\calL^\eps v^\eps(x)\le \ell^\eps$, one obtains that for any $Z$ admissible,
\begin{equation}\label{eq:bound_intro}
\frac1T
\left(\E^{ \Q}[Z_T]+\frac{1}{\eps}D^{\text{KL}}_T( \Q\|\PP)\right)
 \leq
 \frac1Tv^\eps(x) - \frac1T\E^{\Q}[v^\eps(X^Z_T)]+ \ell^\eps .
\end{equation}
Clearly, the first term vanishes as $T\to\iy$. The proof of \cite{alv-hen2019} assumes that $v^\eps$ is bounded below, hence deducing that the second term also vanishes, and the proof is complete. However, as we show in Remark \ref{rem:AH} the function $v^\eps$ is unbounded below in the Verhulst--Pearl diffusion case given in \cite[Section 4.1]{alv-hen2019}, which is the most celebrated example in population dynamics. 
The explosion of the potential function at $x=0+$ stems from the fact that (in a consistent way with the population dynamics literature) the diffusion term vanishes as $x\to 0+$. Moreover, recall that  our arbitrary control is singular, hence it can push the process $X^Z$ very close to zero instantly, which leads to exploding values of $v^\eps(X^Z_T)$. To bypass this issue, one may be tempted to truncate the potential function $v^\eps$ (or its derivatives). However, this leads to a non-negligible violation of the HJB equation (in the sense that as the truncation level goes to infinity, the violation of HJB does not go to zero). We take advantage of the preliminary results established for the existence of a smooth solution to the HJB via the shooting method. Specifically, we consider a truncated version of a {\it perturbed} version of $v^\eps$ by considering a sequence of solutions to ODEs that are associated with threshold-contros whose threshold-levels converge from below to the candidate level $\beta^\eps$. For this sequence, the violation of the HJB vanishes as the threshold converges to $\beta^\eps$. The proof ends by taking first $\liminf_{T\to\iy}$ for each function in the sequence, and then the limit through the sequence of functions.
%

\subsection{Summary and main contributions}\label{sec:14}
In summary, our main contributions are as follows

\begin{itemize}

\item  We provide and solve an ergodic and singular control problem with ambiguity that arises naturally in the harvesting literature. This problem is formulated as a game between a DM and an adverse player.
\item We solve a relevant free-boundary problem and use it to characterize the optimal policy for the DM, which has a natural and simple form. 
\item We correct a gap that appeared in the upper bound argument of 
a previous harvesting paper \cite{alv-hen2019} that looked at the ergodic risk-neutral setting. We establish the upper bound for both the risk-neutral and the ambiguity case.
\item We analyze the dependency of the optimal payoff and optimal policy for the DM on the parameters of the problem.
\end{itemize}

\subsection{Organization}\label{sec:15} The rest of the paper is organized as follows. In Section \ref{sec:2} we set up the stochastic model, provide the underlying assumptions and Theorem \ref{thm:main}, which is the main result of the paper. The proof of the throrem relies on the four propositions given in Section \ref{sec:3}.
Section \ref{sec:4} is devoted to the proofs of the previously mentioned propositions. It includes some preliminary ODE results which are summarized in several lemmas. Finally, Section \ref{sec:6} provides comparative statics with respect to the ambiguity parameter $\eps$.

\subsection{Notation}\label{sec:16}
We use the following notation. For $a, b \in \R$, we define $a\wedge b: = \min\{a, b\}$ and $a \vee b := \max\{a, b\}$. We use $\R_+$ to denote $[0,\infty)$. We denote by $\calC^1$ or $\calC^2$ the sets of functions with continuous first derivatives or continuous second derivatives. By RCLL we mean right-continuous with finite left limits. For any Borel set $A$, $\one_A$ is the indicator function of $A$: $\one_A(x)=1$ if $x\in A$ and $\one_A(x)=0$ if $x\notin A$.

%
%
%
%
%
%
%

\section{The stochastic model and the main result}\label{sec:2}
In this section we describe the ergodic harvesting problem with ambiguity. We start with a rigorous definition of the control problem as a two-player game, setting up the set of admissible controls for the players. Then, we introduce the payoff function and a set of candidate optimal controls for the DM. A relevant free boundary problem is provided. We intuitively explain how it is associated with the optimal control and the value. Finally, we introduce the assumptions on the model and state the main result of the paper.

\subsection{Dynamics and controls}\label{sec:21}
The rigorous definition of the control problem with ambiguity is now given. 
Consider a filtered probability space $(\Omega,\calF,\{\calF_t\},\PP)$ that supports
 a one-dimensional Wiener process $  W$ adapted to the filtration $\{\calF_t\}$ (satisfying the usual conditions)
and the process
\begin{align}\label{dynamics_0}
X_t=x+\int_0^tX_s\mu(X_s)ds+\int_0^t\sigma(X_s) dW_s,\quad t\in\R_+,
\end{align}
which represents the population size in the absent of harvesting. The functions $\mu$ and $\sigma$ satisfy some conditions given in Assumptions \ref{assumptions:absorption} and \ref{assumptions:main} below. The value $\mu(X_t)$ stands for the per-capita growth rate and $\sigma^2(X_t)/X_t^2$ is the infinitesimal variance of fluctuations in the per-capita growth rate.

A fundamental assumption that is in force throughout the paper is that the population size does not explode and does not go extinct in a finite time. For this we need the following definitions. Fix an arbitrary $c>0$. 
The density of the {\it scale function} of the unharvested process $X$ from \eqref{dynamics_0} under the probability measure $\PP$ 
is given by
\begin{align}\label{scale}
S_{\PP}'(x)=\exp\left(-\int_c^x\frac{2\mu(y)y}{\sigma^2(y)}dy\right),\qquad x \in(0,\iy).
\end{align}
\begin{assumption}\label{assumptions:absorption} The following hold:
\begin{enumerate}[(1)]
\item[(A0)]
\[
 \lim_{y\to 0+}S_\PP(y):= \lim_{y\to 0+} S_\PP((y,c))=\int_y^c S_\PP'(x)dx =-\infty,\qquad \lim_{y\to \iy} S_\PP(y)=\infty.
\]
\end{enumerate}
\end{assumption}
This assumption ensures that the SDE \eqref{dynamics_0} does not explode, has a pathwise unique solution, and that in many examples the unharvested process is positive recurrent and converges to its unique invariant probability measure. It is necessary to assume that $0$ is a boundary that cannot be attained in finite time by the unharvested diffusion. Moreover, it also natural to assume that $\PP(\lim_{t\to \infty} X_t =0)=0$, since otherwise the harvest yield might be zero. This implies that, following the boundary classification due to Feller \cite{KT81}, $0$ has to be either an entrance or a natural non-attracting boundary. This happens if and only if $S_\PP(0)=-\infty$ (see table 6.2 from \cite{KT81}).
\begin{remark}
We note that the related 
work of \cite{alv-hen2019}  
has an additional condition that ensures the speed measure is finite. This is done in order to make sure that the unharvested diffusion has a stationary distribution. We do not require this condition as our method of proof does not require ergodicity - we focus on ODE methods to explore the control problem.
\end{remark}

\begin{definition}[Admissible Controls]\label{def:admissible}
\begin{enumerate}[(i)]
\item An {\it admissible control for the DM} for any initial state $  x>0$ is
a nondecreasing process $ Z=(Z_t)_{t\in\R_+}$ taking values in $\R_+$ with RCLL sample paths adapted to the filtration $\{\calF_t\}$, such that the dynamics $(X^Z_t)_{t\in\R_+}$, satisfies,
\begin{align}\label{dynamics}
X^Z_t=x+\int_0^tX^Z_s\mu(X^Z_s)ds+\int_0^t\sigma(X^Z_s) dW_s- Z_t,\quad t\in\R_+,
\end{align}
with $  X^Z_t\ge 0$, $t\in\R_+$, $\PP$-almost surely (a.s.). The functions $\mu$ and $\sigma$ are measurable and satisfy some conditions, provided in Assumption \ref{assumptions:main} in the sequel.
	
	\item An {\it admissible control for the adverse player} is a measure $\Q$ defined on $(\Omega,\calF,\{\calF_t\})$,
\begin{align}\label{RN}
\frac{d \Q}{d\PP}(t)=\exp\left\{\int_0^t \theta(X^Z_s) d W_s-\frac{1}{2}\int_0^t  (\theta(X^Z_s))^2ds\right\},\quad t\in\R_+,
\end{align}
for a function $\theta:\R_+\to\R$, satisfying 
\begin{align}\label{eq:psi_cond}
\quad\E^{\PP}\left[e^{\frac{1}{2}\int_0^t (\theta(X^Z_s))^2ds}\right]<\iy,\quad t\in\R_+,
\end{align}
and such that the conditions in (A0) hold for $S_\Q'(x):=\exp\left(-\int_c^x\frac{2\mu(y)y+\sigma(y)\theta(y)}{\sigma^2(y)}dy\right)$. 
In the sequel, we refer to $\psi_t=\theta(X^Z_t), t\in\R_+,$ as the {\rm Girsanov kernel} of $\Q$.
\end{enumerate}
\end{definition}
We denote by $ \calZ( x)$ the set of all admissible controls for the DM, given the initial condition $x$. The set of all admissible controls for the adverse player is denoted by $ \calQ( x)$.

\begin{remark}\label{rem:strat}
One can write the dynamics from \eqref{dynamics} in the alternative form
\begin{align}\label{dynamics:Q}
X^Z_t= x+\int_0^t X^Z_s\mu(X^Z_s)ds+\int_0^t\sigma(X^Z_s)\psi_sds+ \int_0^t\sigma(X^Z_s)  dW^{\Q}_s- Z_t,\quad t\in\R_+,
\end{align}
where $ W^{ \Q}_t:= W_t-\int_0^t\psi_sds$, $t\in\R_+$, is an $\{\calF_t\}$-one-dimensional $\Q$-Wiener process. Under (A0) it follows that $\PP(X_t>0, t>0|X_0=x)=1, x>0$. By the definition of admissible controls $\Q\in\calQ(x)$, this condition is in force also under $\Q$; it implies that the adverse player cannot change the measure in a way that would lead to a finite-time extinction. The player has the strong belief that the population should not go extinct in a finite time. From a biological standpoint this restriction on the measures $\Q$ is very natural - we restrict ourselves to a reasonable neighborhood of the measure $\PP$, one where there are no extinctions.
\end{remark}

\subsection{The payoff function}\label{sec:22} 
\
Fix a parameter $\eps>0$ to which we refer to as the \textit{ambiguity parameter}. The ergodic expected payoff associated with the initial condition $x$ and the controls $Z$ and $ \Q$ is given by
\begin{align}\notag
J^\eps( x, Z, \Q):=\liminf_{T\to\iy} \frac{1}{T}\left\{\E^{ \Q}[Z_T]+\frac{1}{\eps}D^{\text{KL}}_T( \Q\|\PP)\right\}
,
\end{align}
where
\begin{align}\notag
%
&D^{\text{KL}}_T( \Q\|\PP):=\E^{ \Q}\left[\int_0^T\ln\left(\frac{d \Q}{d\PP}(t)\right)dt\right]
\end{align}
is the Kullback--Leibler divergence. The payoff function can be reformulated in the technically more convenient form
\begin{align}\label{cost}
J^\eps( x, Z, \Q)=&\liminf_{T\to\iy}\frac{1}{T}
\E^{ \Q}\Big[\int_0^T \Big(  d Z_t+\frac{1}{2\eps} \psi^2_tdt\Big) \Big],
\end{align}
where $ \psi$ is the Girsanov kernel of $\Q$.

The risk-neutral (no ambiguity) payoff is given by
\begin{equation}\label{no_amb}
\begin{split}
J^{0}( x, Z)&:=\liminf_{T\to\iy}\frac{1}{T}\E^{ \PP}\left[\int_0^{T} d Z_t \right]
.
\end{split}
\end{equation}
For comparison reasons we place the risk-neutral and the ambiguity models under the same umbrella. So, in our general setting, the risk-neutral payoff is associated with $\eps=0$ (we justify this in Remark \ref{rem:ambiguity} below). For any $\eps\ge 0$, we define the value function by
\begin{equation}\label{e:value}
V^{\eps}( x)=
\begin{cases}
\sup_{Z\in \calZ( x)}\;\inf_{ \Q\in \calQ( x)}\;J^\eps( x, Z, \Q), &\eps>0,\\
\sup_{Z\in \calZ( x)}\;J^0( x, Z), &\eps=0.
\end{cases}
\end{equation}An admissible control $Z$ is called an {\it optimal control} if it attains the value function, that is $V^\eps(x)=\inf_{ \Q\in \calQ( x)}\;J^\eps( x, Z, \Q)$ and in case $\eps=0$, $V^0(x)=J^0( x, Z)$.


\begin{remark}\label{rem:ambiguity}
Here we explain some of the intuition behind the game structure and explain it from a biological point of view. In the natural world we do not know the true model so we do not know the measure $\PP$. We therefore use a measure $\Q$ that we hope is close to $\PP$. The closeness of the measures is given by $D^{\text{KL}}_T( \Q\|\PP)$.
The intuition behind the payoff $J^\eps( x, Z, \Q)$ is the following. The term $\E^{ \Q}[Z_T]$ is the expected value under the measure $\Q$ of the total harvest between $0$ and $T$. The second term $\eps^{-1}D^{\text{KL}}_T( \Q\|\PP)$ is the penalization due to using the measure $\Q$ instead of the real measure $\PP$.
Informally, note that for large (small) values of $\eps$, the {\it penalty term} $(1/\eps)D^{\text{KL}}_T( \Q\|\PP)$ allows for large (small) values of the divergence. This means that the adverse player is more (less) likely to choose $\Q$'s that are farther away from $\PP$. In other words, larger (smaller) values of $\eps$ correspond to larger (smaller) level of ambiguity. As we show in Section \ref{sec:6}, as $\eps\to0+$, the penalty term averages out to zero and the problem convergence to the risk-neutral one.
\footnote{Yet, it is not so obvious that the $\eps^{-1}D^{\text{KL}}_T( \Q\|\PP)\to 0$ as $\eps\to 0+$. Indeed, the first term converges to $\iy$, while the second to $0$. One needs to show that the rate of convergence of the second term is faster. This is done in
Theorem \ref{thm:valuecont}.}
We divide the sum of the two payoff components by the time horizon $T$ and let $T\to \iy$ to get the \textit{penalized asymptotic yield} $J^\eps( x, Z, \Q)$. The optimization problem, becomes the following: the DM chooses a control $Z$ and the adverse player picks an (open loop) control $\Q$ in response, which is adapted to the same underlying filtration $\calF_t$. This control aims to be the worst possible measure for the DM, while the adverse player's hands are tied due to the divergence penalty term, and he is forced to choose a measure in an `$\eps$-neighborhood' of $\PP$.
%
%
\end{remark}

\subsection{Candidate controls for the DM: thresholod controls}\label{sec:23}
The ergodic control problem without ambiguity \eqref{no_amb} was studied by Alvarez and Hening in
\cite{alv-hen2019}. They proved that the optimal control for the DM is one that uses minimal effort to keep the population in a given interval of the form $[0,\beta]$, where $\beta$ depends on the parameters of the problem. Our main result shows that these types of controls are also optimal in the more general setting that includes an ambiguity.

To rigorously define such a control we make use of the {\it Skorokhod map on an interval}. Fix $\beta>0$. For any $\eta\in\calD(\R_+,\R)$ there exists a unique couple of functions $(\chi,\ph)\in\calD(\R_+,\R^2)$ that satisfies the following properties:
\noi
(i) for every $t\in\R_+$, $\chi(t)=\eta(t)-\ph(t)$;
\noi
(ii) $\ph$ is nondecreasing, $\ph(0-)=0$, and
\begin{align}\notag
	\int_0^\iy
	\one_{(-\iy,\beta)}(\chi(t))d\ph(t)=0.
\end{align}
We define $\Gamma_{\beta}[\eta]:=(\Gamma_{\beta}^2
,\Gamma_{\beta}^2)[\eta]=(\chi,\ph)$.
See \cite{Kruk2007} for the existence and uniqueness of solutions,
as well as the continuity and further properties of the map.

\begin{definition}\label{def_Skorokhod}
Fix $x,\beta\in[0,b]$. The control $Z=Z^{(\beta)}$ is called
a $\beta$-threshold control if for every $\eta\in \calC(\R_+,\R)$ one has $(X^Z, Z)(\eta)=\Gamma_{\beta}[\eta]$.
\end{definition}
\noindent One can easily verify that any $\beta$-threshold control is admissible in the sense of Definition \ref{def:admissible}.

\subsection{The free-boundary problem.}\label{sec:24}

We show that for any $\eps \geq 0$, there are positive constants $\ell^\eps$ and $\beta^\eps$, 
such that for any initial state $x>0$,
\begin{align}\notag
 \text{$V^\eps(x)=\sup_{\Q\in\calQ(x)}J^\eps(x,Z^{(\beta^\eps)},\Q)=\ell^\eps$ in case $\eps>0$, and 
$V^0(x)=J^0(x,Z^{(\beta^0)})=\ell^0$ in case $\eps=0$.}
\end{align}
%

As in \cite{alv-hen2019}, this suggests that the value $\ell^\eps$ and the threshold level $\beta^\eps$ can be characterized by an HJB equation that has the form of a free boundary problem with two parts. Motivated by the game structure, together with the dynamics and payoff forms given in \eqref{dynamics:Q} and \eqref{cost}, for any $\eps\ge 0$, let $\calL^\eps$ be the operator which acts on $f\in \calC^2$ as
\begin{align}\label{operator}
\begin{split}
\calL^\eps f(x)&:=\inf_{p\in\R}\Big\{\frac{1}{2}\sigma^2(x)f''(x)+(x\mu(x)+\sigma(x)p)f'(x)+\frac{1}{2\eps}p^2
\Big\}
\\
&=
\frac{1}{2}\sigma^2(x)f''(x)+x\mu(x)f'(x)-\frac{\eps}2\sigma^2(x)(f'(x))^2.
\end{split}
\end{align}
While the representation on the first line is not valid for $\eps=0$, the second one is valid for any $\eps\ge 0$ and coincides with the operator in the risk-neutral case, see \cite[equation (4)]{alv-hen2019}.

The relevant HJB equation is given in \eqref{HJB1}. However, follwoing our educated guess that the optimal control for the DM is a threshold-control, we choose to work with the following more explicit free bounday ODE. Namely,
we are looking for the maximal $\ell$ for which there exists $f\in \calC^2$ and a number $\hat x>0$ such that
\begin{align}\label{HJB_cutoff}
\begin{cases}
\calL^\eps f(x)=\ell,\quad f'(x)\ge 1, &x\in(0,\hat x],\\
\calL^\eps f(x)\le \ell,\quad f'(x)=1, &x\in(\hat x,\iy),
\end{cases}
\end{align}
as well as $f''(\hat x)=0$ and $f'(\hat x)=1$. In particular,
\begin{align}\label{l_eps_1}
\Leps(\hat x)=\ell,
\end{align}
where,
\begin{align}\label{l_eps_2}
\Leps(x):= x\mu(x)-\frac{\eps}{2}\sig^2(x),\qquad x\in (0,\infty).
\end{align}
\begin{definition}\label{def:23}
We denote by $(\hat x, \ell_{\hat x}, f)$ a {\it solution of \eqref{HJB_cutoff}--\eqref{l_eps_1}}, where $\ell_{\hat x}:=\Leps(\hat x)$. A solution $(\hat x, \ell_{\hat x}, f)$ is called an {\rm optimal solution of \eqref{HJB_cutoff}--\eqref{l_eps_1}} if for any other solution of \eqref{HJB_cutoff}--\eqref{l_eps_1}, say $(\hat y, \ell_{\hat y}, h)$, one has $\ell_{\hat y}\le \ell_{\hat x}$. We refer to $f$ as the {\rm potential function}.
\end{definition}

The rationale behind this is as follows: when the initial population size is $X^Z_0=x\in(\hat x,\iy)$, then in order to keep the process $X^Z$ between $0$ and $\hat x$, there is an instantaneous harvesting of size $x-\hat x$. When $x\in(0,\hat x)$ no action is being taken by the DM. When $X^Z$ hits the boundary $\hat x$, the threshold policy is taking action, leading to the Neumann boundary condition at $\hat x$. The population size will be kept in $(0, \hat x)$, with an initial harvest $0\vee (X^Z_0 - x)$ and then with harvesting only when the population size $X^Z$ is at level $\hat x$.

\subsection{Further assumptions and the main result}\label{sec:26}
We now present the second set of assumptions that hold throughout the paper.
\begin{assumption}\label{assumptions:main} The following hold:
\begin{enumerate}[(1)]
 \item[(A1)] The function $\sig:(0,\iy)\to(0,\iy)$ is increasing and continuously differentiable on $(0,\infty)$. Moreover,
 $\sig'$ is nondecreasing and bounded by a constant $\sig_0$. The function $\mu:[0,\iy)\to(0,\iy)$ is continuously differentiable on $(0,\infty)$ and the function $\mu'$ is bounded as $x\to 0+$. The function $x \mapsto \frac{x\mu(x)}{\sig(x)}$ is decreasing and bounded as $x \to 0+$.
We also assume that there exist numbers $\bar \sig, c ,\bar\mu>0$ such that for sufficiently small $x$, $|\sig(x) - \bar \sig x|\leq c x^2$ and $|\mu(x) - \bar \mu | \leq cx$.

\item[(A2)]  There exists $x^\eps\in(0,\iy)$ such that the function $\Leps$ defined in \eqref{l_eps_2} is increasing on $(0,x^\eps)$ and decreasing on $(x^\eps,\infty)$. Also, assume $\bar x^\eps:=\inf\{x\ge x^\eps:\Leps(x)=0\}$ is finite.

%
\end{enumerate}
\end{assumption}

Before discussing the assumption, we show that the assumptions hold in the most celebrated example in population dynamics, also referred to as the {\it Verhulst--Pearls diffusion} or the {\it logistic diffusion} model.
\begin{example}[Verhulst--Pearl diffusion]\label{ex:VP}
In this setting, the dynamics \eqref{dynamics_0} of the unharvested population is given by
\begin{align}\notag
dX_t=\bar\mu X_t(1-\bar \gamma X_t)dt+\bar \sigma X_tdW_t,\qquad t\in\R_+,
\end{align}
where $\bar\mu>0$ is the per-capita growth rate, $1/\bar\gamma>0$ is the carrying capacity, and $\bar\sigma^2$ is the infinitesimal variance of fluctuations in the per-capita growth rate. 
One can easily verify that Assumptions (A0)--(A2) hold in this example for any $\eps\in[0,\iy)$. In this case,
\begin{align}\notag
x^\eps=\frac{\bar \mu}{2\bar\mu\gamma+\eps\bar\sigma^2}\qquad\text{and}\qquad\bar x^\eps=2x^\eps.
\end{align}
In \cite{alv-hen2019}, where a finite speed measure is required, it is also necessary that the long-term behavior of the unharvested system, which is characterized by the stochastic growth rate
$r:=\bar \mu - \frac{\bar \sigma^2}{2}$ is positive. 
\end{example}

We now comment on the assumptions.

\begin{remark}
Part (A1) tells us that the functions $\sigma$ and $\mu$ are well-behaved and the diffusion is non-degenerate. In addition, there are some technical assumptions on the regularity, boundedness and monotonicity of the coefficients $\mu$ and $\sigma$. We note that these assumptions are similar to those from \cite{alv-hen2019} and \cite{JZ06}. The extra second order bounds around $x=0+$ compared to \cite{alv-hen2019} ensure that we can side-step the gap from the proof in \cite{alv-hen2019} - see Remark \ref{rem:AH}. Part (A2) here is the generalization of Assumption 2.2 (A2) from \cite{alv-hen2019} to the setting that includes ambiguity. In particular, this is natural in ecological applications: initially, at low densities the competition for resources is weak so the growth rate grows from $0$ at $0$ up to a maximal value, after which, due to competition, the growth rate decreases to $0$ and finally becomes negative.

In all biological applications we will have $\sigma(0)=\lim_{x\to 0+}\sigma(x)=0$ because the population cannot escape $0$ if it starts at $0$ - an extinct population will not get `resurrected'. Moreover, for most applications the natural choice is $\sigma(x) = \bar \sigma x$ for some $\bar \sigma>0$. In a biological setting it will also be natural to have that the unharvested population $X$ given by \eqref{dynamics_0} has a stationary distribution.

\end{remark}

We next present our main result. We prove that an optimal control for the DM exists and that it is a threshold control. Moreover, we show that the threshold level and the value function $V^{\eps}$ are characterized by the free-boundary problem \eqref{HJB_cutoff}--\eqref{l_eps_1}. Finally, we show that the population dynamics are stationary under the measure chosen by the adverse player. Even though our main interest is when there is ambiguity, i.e., when $\eps>0$, we consider also the risk-neutral case $\eps=0$ since there is a gap in the analysis 
from \cite{alv-hen2019}, which is filled in Proposition \ref{prop:upper} below; see also Remark
\ref{rem:AH} 
below. The proof of the Theorem is given in the next section.
\begin{theorem}[Main Theorem]\label{thm:main}
For any $\eps \geq 0$, the following hold:
\begin{enumerate}
\item There exists an optimal solution to \eqref{HJB_cutoff}--\eqref{l_eps_1}: $(\beta^\eps,\ell^\eps,v^\eps)$ and $\beta^\eps\in (x^\eps, \bar x^\eps)$.
\item The $\beta^\eps$-threshold control, denoted by $Z^{\eps}:=Z^{(\beta^\eps)}$, is optimal.
\item For any initial state $x>0$, the value of the problem is $\ell^\eps$. In other words, if $\eps>0$ then
\[
V^{\eps}( x)=\sup_{Z\in \calZ( x)}\;\inf_{ \Q\in \calQ( x)}\;J^\eps( x, Z, \Q)= \ell^\eps,\qquad\text{for any $x\in(0,\iy)$,}
\]
and if $\eps=0$ then
\[
V^{0}(x)={\sup_{Z\in\calZ(x)}}J^{0}( x, Z)=J^{0}( x, Z^0)=\ell^0,\qquad\text{for any $x\in(0,\iy)$}.
\]
\end{enumerate}
\end{theorem}

\section{Proof of Theorem \ref{thm:main}.}\label{sec:3}
The proof of the main theorem relies on Propositions \ref{prop:HJB}--\ref{prop:upper} given below. For completeness, we provide the derivation of Theorem \ref{thm:main} at the end of this section. The proofs of the four propositions are given in Section \ref{sec:4}.

The main component in characterizing the value and the optimal control for the adverse player is via an optimal solution $(\beta^\eps,\ell^\eps,v^\eps)$ to \eqref{HJB_cutoff}--\eqref{l_eps_1}. The next proposition establishes the existence of such an optimal triplet. Note that the quantities $x^\eps, \bar x^\eps$ were defined in Assumption (A2).
\begin{proposition}\label{prop:HJB}
For any $\eps\ge 0$ there exists an optimal solution to \eqref{HJB_cutoff}--\eqref{l_eps_1} with $\beta^\eps\in(x^\eps,\bar x^\eps)$.
\end{proposition}
Let $(\beta^\eps, \ell^\eps, v^\eps)$ be the optimal  solution of \eqref{HJB_cutoff}--\eqref{l_eps_1} given in Proposition \ref{prop:HJB}. The next proposition is needed for technical reasons in order to prove Proposition \ref{prop:lower} below and to obtain comparative statics (see Section \ref{sec:6}). 

\begin{proposition} \label{prop:bounded}
For any $\eps\ge 0$, the function $\sig(\cdot) (v^\eps)'(\cdot)$ is bounded above by $\sig(\beta^\eps)$ on $(0,\beta^\eps]$. 
\end{proposition}

The next proposition states that by using the $\beta^\eps$-threshold control, $Z^\eps:=Z^{(\beta^\eps)}$, the DM attains at least the value $\ell^\eps$
. In particular, it provides a lower bound for the value.
\begin{proposition}\label{prop:lower}
For the optimal solution of \eqref{HJB_cutoff}--\eqref{l_eps_1} $(\beta^\eps, \ell^\eps, v^\eps)$, one has
\begin{align}\notag
\ell^\eps=\inf_{\Q\in\calQ(x)}J^\eps(x,Z^\eps,\Q),\qquad x\in(0,\iy),
\end{align}
for $\eps>0$, and 
\begin{align}\notag
\ell^0=J^0(x,Z^\eps),\qquad x\in(0,\iy),
\end{align}
for $\eps=0$.
\end{proposition}
Finally, the next proposition provides an upper bound for the value. 
\begin{proposition}\label{prop:upper}
For the optimal solution of \eqref{HJB_cutoff}--\eqref{l_eps_1} $(\beta^\eps, \ell^\eps, v^\eps)$, one has
\begin{align}\notag
\sup_{Z\in\calZ(x)}\inf_{\Q\in\calQ(x)}J^\eps(x,Z,\Q)\le \ell^\eps,\qquad x\in(0,\infty),
\end{align}
for $\eps > 0$,
and
\begin{align}\label{eq:upper:bound1}
 \sup_{Z\in\calZ(x)}J^0(x,Z)\le \ell^0,\qquad x\in(0,\infty),
\end{align}
 for $\eps = 0$.
\end{proposition}

\begin{remark}\label{rem:AH}
We note here that the proof in the $\eps=0$ case appearing in \cite{alv-hen2019} has a gap. This is because in \cite{alv-hen2019} the authors assumed that $v^0=v^{\eps=0}$ is bounded below, something which is not always true. Take, for example, the Verhulst--Pearl diffusion from Example \ref{ex:VP} with parameters $\mu = \sig = \gam = 1$. Then,  one can show that $(v^0)' (x)= c(e^{2x} -1)/x^2$,
which behaves like $x^{-1}$ in the neighborhood of $x=0+$. As a result, $v^0(x) \to -\iy$ as $x \to 0+$. Specifically, in Lemma 2.1 from \cite{alv-hen2019}, where the authors aimed at showing an upper bound as in \eqref{eq:upper:bound1}, they applied It\^o's rule to test functions (candidates for $v^0$) and in the proof they assumed that these functions are bounded from below. Then, they applied it in their Theorem 2.1 for $v^0$ (which is not always bounded below). We on the other hand consider bounded from below test functions, yet we allow them to slightly violate the second part of \eqref{HJB_cutoff}. The main difficult is to choose a proper sequence of functions such that in the limit the violation vanishes.
\end{remark}

\begin{proof}[Proof of Theorem \ref{thm:main}] The case $\eps=0$ is complete by replacing \cite[Lemma 2.1]{alv-hen2019} with Proposition \ref{prop:upper}. 
Hence, in the rest of the proof we fix an arbitrary $\eps>0$.

Proposition \eqref{prop:HJB} establishes bullet (1). 
From Proposition \ref{prop:upper}, we know $$\ell^\eps \ge\sup_{Z\in \calZ( x)}\;\inf_{ \Q\in \calQ( x)}\;J^\eps( x, Z, \Q).$$
On the other hand, from Proposition \ref{prop:lower}, if we take the $\beta^\eps$-threshold control, $Z^\eps$, we get
$$\ell^\eps=\inf_{\Q\in\calQ(x)}J^\eps(x,Z^\eps,\Q)
\leq \sup_{Z\in \calZ( x)}\;\inf_{ \Q\in \calQ( x)}\;J^\eps( x, Z, \Q).$$ As a consequence: $ V^\eps(x)\equiv \ell^\eps $ (establishing (3)) and $Z^\eps$ is an optimal control (establishing (2)). 


\end{proof}

\section{Proof of Propositions \ref{prop:HJB}, \ref{prop:bounded},  \ref{prop:lower}, and \ref{prop:upper}}\label{sec:4}

The proofs of Proposition \ref{prop:HJB} and \ref{prop:bounded} require some preliminary ODE results, provided in a sequence of lemmas below. On the other hand, the proof of Proposition \ref{prop:lower} merely requires the existence of a $\calC^2$ solution to \eqref{HJB_cutoff}, which we get thanks to Proposition \ref{prop:HJB}. For readability reasons we start with the proof of Proposition \ref{prop:lower}.

\subsection{Proof of Proposition \ref{prop:lower}.}\label{sec:41}
We provide the proof for the case $\eps>0$. The case $\eps=0$ is handled similarly and is simpler. Therefore, it is omitted. 
Denote by $Z=Z^{(\beta^\eps)}$ the $\beta^\eps$-threshold control. Choose an arbitrary admissible control $\Q$ with Girsanov kernel $\psi$. For every $n\in\N$, set $T^n=T\wedge\inf\{t\ge 0: X^Z_n\notin[1/n,n]\}$.
By It\^o's Lemma,
\begin{align*}
  v^\eps(X^Z_{T_n}) &= v^\eps(x) + \int_0^{T_n} \Big(\frac 12 \sigma^2(X^Z_s)({v^\eps})^{''}(X^Z_s)+\big[X^Z_s\mu(X^Z_s) + \sigma(X^Z_s)\psi_s]({v^\eps})'(X^Z_s)\Big)ds
 \\
 &\quad+ \int_0^{T_n} \sigma(X^Z_s)({v^\eps})'(X^Z_s)dW^\Q_s - \int_0^{T_n} ({v^\eps})'(X^Z_s)dZ_s.
\end{align*}
The function $v^\eps$ solves \eqref{HJB_cutoff}. From \eqref{operator} (which is valid for $x>0$) it follows that
\begin{align*}
&\frac 12 \sigma^2(X^Z_s)({v^\eps})^{''}(X^Z_s)+[X^Z_s\mu(X^Z_s) + \sigma(X^Z_s)\psi_s]({v^\eps})'(X^Z_s)+ \frac{1}{2\eps}\psi_s^2 \ge \ell^\eps.
\end{align*}
Hence,
\begin{align*}
  \int_0^{T^n} ({v^\eps})'(X^Z_s)dZ_s + \int_0^{T^n}  \frac{1}{2\eps}\psi_s^2ds \geq v^\eps(x) - v^\eps(X^Z_{T^n})+ \ell^\eps {T^n} + \int_0^{T^n} \sigma(X^Z_s)({v^\eps})'(X^Z_s)dW^\Q_s.
\end{align*}
Since $Z$ is a $\beta^\eps$-threshold control, $\int_0^\infty 1_{[0,\beta^\eps)}(X^Z_s)dZ_s = 0$, so ${v^\eps}'(X^Z_s) = 1 $ when $X^Z_s = \beta^\eps$ gives $\int_0^{T_n} {v^\eps}'(X^Z_s)dZ_s = \int_0^{T_n}dZ_s.$
Taking expectation with respect to $\Q$ and noting that by Proposition \ref{prop:bounded}
$$\E\left[\int_0^{T_n} \sigma(X^Z_s)({v^\eps})'(X^Z_s)dW^\Q_s\right]=0$$
imply that
$$
   \frac{1}{T_n}\mathbb E^{\Q} \Big[\int_0^{T_n}dZ_s + \int_0^{T_n} \frac{1}{2\eps}\psi_s^2ds\Big] \geq \frac{1}{T_n}\left(v^\eps(x) - \mathbb E^{\Q} [v^\eps(X^Z_T)]\right) + \ell^\eps \geq \frac{1}{T_n}\left(v^\eps(x)-v^\eps(\beta)\right) + \ell^\eps,
$$
where the last inequality follows by the monotonicity of $v^\eps$. Letting first $n\to\iy$ then by the admissibility of $\Q$ and the fact that $Z$ is a threshold policy, it follows that $\Q(X^Z_s>0,\;s\in[0,T])=1$ (see e.g., \cite[Section 2.6]{BS12}), which implies that $T^n\to T$, $\Q$-a.s. Then take $T \to \infty$ and get $J^\eps(x,Z,\Q) \geq \ell^\eps$. Since $\Q$ is arbitrary admissible, one has, $$\ell^\eps = \inf_{\Q \in\mathcal Q (x)}J^\eps(x,Z,\Q).$$
\qed

\vskip10pt

\subsection{Proof of Proposition \ref{prop:HJB}}
The proof uses the shooting method. This is a method that allows to solve boundary value problems by reducing them to initial value problems; see \cite{Stoer1980} for further reading. We adapt it to our free-boundary problem. In our case, we set up as a parameter the boundary point $\hat x=\beta^\eps$ such that \eqref{HJB_cutoff} holds true with it and such that it gives the maximal value $\ell^\eps$.

We start by setting up an ODE that stems from \eqref{HJB_cutoff}. Its role is described in two paragraphs ahead.  Fix $b>0$ and recall the function $\Leps$ from \eqref{l_eps_2}. 
Denote by $g_{b,\gamma}$ the $\calC^1((0,b)\cup(b,\iy))\cap \calC(0,\iy)$ solution of the following ODE.
\begin{align}\label{eq:gbgamma}
\begin{cases}
\frac 12\sigma^2(x) g'(x)+x\mu(x) g(x)-\frac{\eps}{2}\sigma^2(x)g^2(x)=\Leps(b)+\gamma,\qquad x\in(0,\infty),
\\
g(b)=1.
\end{cases}
\end{align}
If $\gamma=0$, we use the notation $g_b$ for $g_{b,0}$. Notice that $g_b$ is continuously differentiable at $x=b$. In the rest of the section we will make several uses of this ODE, which originates from \eqref{HJB_cutoff}. The existence and uniqueness of a solution to \eqref{eq:gbgamma} follows by the Cole--Hopf transformation. Indeed, the following ODE is linear, hence, admits a unique solution on any interval of the form $[a_1,a_2]$ with $a_1 < b < a_2$:
\begin{align*}
 \begin{cases}
\frac 12\sigma^2(x) \phi''(x)+x\mu(x) \phi'(x)=-(\Leps(b)+\gamma)\eps\phi(x),\qquad 
\\
\phi(b) = 1,\quad
\phi'(b) = -\eps.
\end{cases}
\end{align*}
Set $f:=\ln(\phi)/\eps$. Then, $f'$ solves \eqref{eq:gbgamma}. Uniqueness holds since the transformation is one to one.

We now motivate the analysis of the system \eqref{eq:gbgamma}. Targeting at proving Proposition \ref{prop:HJB}, we aim at showing that there is $\beta^\eps\in(x^\eps,\bar x^\eps)$ for which \eqref{HJB_cutoff} holds.
There are four conditions embedded in \eqref{HJB_cutoff}.
Set $f(x)=\int_b^xg_b(y)dy$ (so $f'(x)=g_b(x)$) on $(0,b)$. At this point the reader may see that the ODE for $f$ satisfies $\calL^\eps f(x)=\Leps(b)$  to the left of $b$. By Assumption \ref{assumptions:main}, for $b > x^\eps$, $\Leps(x) \leq \Leps(b)$. Setting up $f'(x)=1$ on $(b,\iy)$ and the second line of \eqref{HJB_cutoff} holds for any $b>x^\eps$. This suggests that most of the effort should be invested in choosing a point $\beta^\eps=b$ for which the challenging bound $g_b(x)\ge 1$ holds on $(0,b)$. Of course, one also needs to choose the $b$ leading to the maximal payoff. Using the fact that $\Leps$ is decreasing to the right of $x^\eps$, we are looking for a minimal $b$ with the property mentioned above. Finally, some effort is required to show that  the infimum over a relevant collection of functions $\{g_b(x)\}_{b\in A}$ satisfies \eqref{eq:gbgamma}.


The following elementary lemma will be used several times in the sequel.
\begin{lemma}\label{lemma:deriv}
Let $f$ be a $\calC ^1$ function defined on $(a,b)$. Suppose at $x \in (a,b)$, $f(x) > c$ $($resp., $<c)$ for some $c \in \R$. Then upon existence of $y_1 := \sup\{y \in (a,x): f(y) = c\}$ and $y_2 := \inf\{y \in (x, b): f(y) = c\}$,  $f'(y_1) \geq 0$ $($resp., $\leq 0)$, $f'(y_2) \leq 0$ $($resp., $\geq 0)$.

As a corollary, let $f$ be a $\calC^1$ function defined on $(a,b)$. Fix $x \in (a,b)$ and let $y_1 := \sup\{y \in (a,x): f(y) = f(x)\}$ and $y_2 := \inf\{y \in (x, b): f(y) = f(x)\}$, if they exist. If $f'(x) > 0$ $($resp., $<0)$, then, $f'(y_1) \leq 0$ $($resp., $\geq 0)$ and $f'(y_2) \leq 0$ $($resp., $\geq 0)$.

\end{lemma}

The following lemma is a perturbation result, we use it to get estimations for $g_b$ via estimates of $g_{b,\gamma}$, which are often easier to achieve.
\begin{lemma}\label{lemma:perturb}
For any $b > 0$ and any $y\in(0,b)$, we have $$\sup_{x\in[y,b]}|g_{b,\gamma}(x)- g_{b}(x)|=O(\gamma) \quad\text{as}\quad \gamma\to 0.$$
\end{lemma}
\begin{proof}[Proof of Lemma \ref{lemma:perturb}]
Fix $b > 0$. Set $H: (0,\infty) \times \R \to \R$, by $H(x,y) := \frac{2}{\sig^2(x)}(\Leps(b)-x\mu(x)y+\frac{\eps}{2}\sigma^2(x)y^2)$, and notice that in the case $\gam = 0$, the ODE \eqref{eq:gbgamma} can be rewritten as $g_{b}'(x) = H(x, g_{b}(x))$. Had $H$ been Lipschitz-continuous in its second argument, standard perturbation theory implies that the solutions to the perturbed ODEs, $g_{b,\gam}'(x) = H(x, g_{b,\gam}(x)) + \frac{\gam}{\sig^2(x)}$, converge to $g_b$ uniformly on sets of the form $[y, b]$, $y \in (0, b)$. To this end, fix $y \in (0, b)$. Notice that for fixed $b \in (0,\infty)$, $g_b$ is bounded on $[y, b]$, say by $K = K(y, b)$. Define $k(x) = (-K \vee x)\wedge K$. Then $g_b(x) = k(g_b(x))$ for $x \in (y, b]$. Hence, in case $\gam = 0$, \eqref{eq:gbgamma} is equivalent to $g_b'(x) = H_k(x, g_b(x))$, where $H_k(x,y) = H(x,k(y))$. The function $H_k$ is Lipschitz-continuous in its second argument. Then, since we have $g_{b,\gam}'(x) = H_k(x, g_{b,\gam}(x)) + \frac{\gam}{\sig^2(x)}$, by \cite[Theorem 9.1]{DifEqn} we get $\sup_{x \in [y,b]}|g_{b, \gam}(x) - g_b(x)| = O(\gam)$.
\end{proof}

Set
\begin{align}\label{def:beta}
B^\eps := \{b> 0: g_b(x) \geq 1, x \in (0,b]\}\qquad\text{and define}\qquad\beta^\eps := \inf B^\eps.
\end{align}
%
The following lemma implies that the infimum is taken over a non-empty set. It also provides a region for $\beta^\eps$.

\begin{lemma}\label{lemma:brange}
The following relations hold, $(0,x^\eps]\cap B^\eps=\emptyset$, but $\bar x^\eps \in B^\eps$. As a conclusion, $\beta^\eps$ is well-defined and $x^\eps \leq \beta^\eps \leq \bar x^\eps$.
\end{lemma}
\begin{proof}[Proof of Lemma \ref{lemma:brange}]
First, fix $b \leq x^\eps$ and $\gam > 0$. Recall the definition of $g_{b,\gam}$, given in \eqref{eq:gbgamma}. Plug in $x=b$ into the ODE for $g_{b,\gamma}$  and use $g_{b,\gam}(b)=1$ and the value of $\la^\eps(b)$, note that $\frac 12\sigma^2(x) g_{b,\gam}'(b)=\gamma$. Therefore $g_{b,\gam}'(b) > 0$. We show that for any $x \in (0,b)$ one has $g_{b,\gam}(x)<1$. Arguing by contradiction, assume it does not hold, then together with $g_{b, \gam}(b) = 1$, the following supremum is attained $x_1: = \sup\{x \in (0, b) : g_{b,\gam}(x) = 1\} $. Then, $$\frac 12\sigma^2(x_1) g_{b,\gam}'(x_1)=\ell_{b} - (x_1\mu(x_1) -\frac{\eps}{2}\sigma^2(x_1))+\gamma = \Leps(b) - \Leps(x_1) + \gam > 0.$$ The second equality follows by the definition of $\Leps$ and the inequality follows by Assumption \ref{assumptions:main} (A2) together with $x_1 < b \leq x^\eps$ and $\gam > 0$. Hence, $g_{b,\gam}'(x_1) > 0$, which contradicts Lemma \ref{lemma:deriv}.

What we have got so far is that in the case $b \le x^\eps$, for any $\gam >0$ one has $g_{b,\gam}(x) < 1$ for $x \in (0,b)$. Then by Lemma \ref{lemma:perturb}, $g_{b,\gam}(x) \to g_b(x)$ as $\gam \to 0+$ for fixed $x \in (0,b)$, therefore, $g_b(x) \leq 1$ for $x \in (0,b)$. However, the structure of the ODE tells us that it can not be the case $g_b(x) = 1$ for all $x \in (0,b)$, therefore, $b \notin B^\eps$.

The case $b = \bar x^\eps$ is similar (we pick $\gam < 0$ instead  of $\gam > 0$), and is therefore omitted.
\end{proof}

Now if the point $\beta^\eps$ is isolated in the set $B^\eps$, the proof of Proposition \ref{prop:HJB} is done, since the result holds with $\beta^\eps$. In case $\beta^\eps$ is an accumulation point, we show the pointwise convergence of $g_{b_i}(x)$ for any fixed $x\in (0,b)$ as $b_i \to \beta^\eps+$, which in turn implies that  $\beta^\eps \in B^\eps$. The following four Lemmas mainly serve this role.

\begin{lemma}\label{lemma:passboundary}
For any $b \geq x^\eps$, and $x > b$ we have $g_b(x) \geq 1$.
\end{lemma}
\begin{proof}[Proof of Lemma \ref{lemma:passboundary}]
As in the previous proof, plug in $x=b$ into the ODE for $g_{b,\gamma}$ and use $g_{b,\gam}(b)=1$ and the value of $\la^\eps(b)$ to get $\frac 12\sigma^2(b) g_{b,\gam}'(b)=\gamma$. Therefore $g_{b,\gam}'(b) > 0$. We show that for any $x \in (b,\infty)$ one has $g_{b,\gam}(x) > 1$. Lemma \ref{lemma:perturb} implies that $g_{b}(x) \ge 1$. Arguing by contradiction, assume it does not hold. Set $x_5 := \inf\{x \in (b, \infty): g_{b,\gam}(x) = 1\}$. Then $$\frac 12\sigma^2(x_5) g_{b,\gam}'(x_5)=\Leps(b) - \Leps(x_5)+\gamma > 0,$$ where the inequality follows since,  $\gamma > 0$ and $\Leps(b) - \Leps(x_5) >0$ because $x^\eps \leq b < x_5$ -- see Assumption \ref{assumptions:main} (A2). This implies $g_{b,\gam}'(x_5) > 0$, which contradicts Lemma \ref{lemma:deriv}.
\end{proof}

\begin{lemma}\label{lemma:geqone}
For any $ b > x^\eps$, there exists $y_2 \in (0,x^\eps)$, such that $g_b(x)\ge 1$ for $x\in(y_2,b]$
\end{lemma}
\begin{proof}[Proof of Lemma \ref{lemma:geqone}]

Fix $b > x^\eps$ and $\gamma<0$.
The definition of $\Leps(b)$ implies $g_{b,\gam}'(b) < 0$. Let $y_2 < x^\eps$ be such that $\Leps(y_2) = \Leps(b)$. Such $y_2$ exists by Assumption \ref{assumptions:main} (A2). We show that for any $x \in (y_2, b)$, $g_{b,\gam}(x)>1$. Arguing by contradiction, assume it does not hold. Set $x_3 = \sup\{x \in (y_2, b): g_{b,\gam}(x) = 1\}$. Then, $$\frac 12\sigma^2(x_3) g_{b,\gam}'(x_3)=\Leps(b) - \Leps(x_3)+\gamma < 0,$$ where the inequality follows since $\gamma < 0$ and $\Leps(b) < \Leps(x_3)$ because $y_2 < x_3 < b$ and by Assumption \ref{assumptions:main} (A2). That is, $g_{b,\gam}'(x_3) < 0$, which contradicts Lemma \ref{lemma:deriv}. Hence $g_{b,\gam}(x)>1$ for $x \in (y_2, b)$.

By Lemma \ref{lemma:perturb}, taking $\gam \to 0-$, we get for $x \in (y_2, b)$, $g_b(x) \geq 1$.
\end{proof}

\begin{lemma}\label{lemma:comparison}
Let $b > a \geq x^\eps$. For $ x \in (0,a)$, one has $g_a(x) \leq g_b(x)$ and for $ x \in (b, \infty)$ one has $g_a(x) \geq g_b(x)$.
\end{lemma}
\begin{proof}[Proof of Lemma \ref{lemma:comparison}]
Fix $x^\eps < a < b$. Set $\del := \Leps(a) - \Leps(b) > 0$ (for the inequality, recall (A2)), and define the function $G:(0,a]\to\R$ by $$G(x) := \del^{-1}(g_a(x) - g_b(x)).$$
Notice that $g_b(a) > 1$. Otherwise, by Lemma \ref{lemma:geqone} and since $a>x^\eps$, $g_b(a) = 1$. However $\sig^2(a)g_b'(a)/2 = -\del<0$, and as a consequence $g_b$ is below 1 in a right neighborhood at $a$, which contradicts the fact $x^\eps < a < b$. The function $G(x)$ satisfies,
\begin{align*}
\begin{cases}
\frac 12\sigma^2(x) G'(x)+x\mu(x) G(x)-\eps\sig^2(x)g_b(x)G(x)-\frac{\eps}{2}\del\sigma^2(x)(G(x))^2= 1,\qquad x\in(0,a],
\\G(a)=\frac{1 - g_b(a)}{\del}.
\end{cases}
\end{align*}
From the previous argument we have $G(a) < 0$. We show that for any $x \in (0,a)$ one has $G(x) < 0$. Arguing by contradiction, suppose this does not hold. Set $x_a = \sup\{x \in (0, a): G(x) = 0\}$. Then $G'(x_a) = 2/\sig^2(a) >0$, in contradiction to Lemma \ref{lemma:deriv}. Therefore, for any $x \in (0,a)$ one has $$g_a(x) = g_b(x) + \del G(x) \leq g_b(x).$$
The part of the reversed inequality for $x \in (b,\infty)$ is similar and uses Lemma \ref{lemma:passboundary} instead of Lemma \ref{lemma:geqone} to get that $G(b) > 0$, and is therefore omitted.
\end{proof}

The following lemma establishes the pointwise convergence as the boundary converges. Therefore, we can use properties of $g_{b_i}$ along a converging sequence $b_i \to b$ to establish properties of $g_b$.
\begin{lemma}\label{lemma:continuity}
For any $b > x^\eps$ and $y < b$, we have
$$|g_{b+\del}(y)-g_b(y)| \to 0\qquad\text{as}\qquad \delta\to 0.$$
Moreovoer, the above holds also for $b = x^\eps$ when $\delta\to0+$.
\end{lemma}
\begin{proof}[Proof of Lemma \ref{lemma:continuity}]
The proof in the case $\del \to 0-$ is similar to the proof in the case $\del \to 0+$. We therefore omit it.

Fix $y\in(0,b)$. By definition, for any $w \in [y,b]$ we have
\begin{align*}
g_{b+\del}(w) - g_b(w)&= g_{b+\del}(b) - g_b(b)
\\
&\quad- \int_y^{b}\frac{2}{\sig^2(x)}\Big[(\Leps(b+\delta)-\Leps(b))\\
&\qquad\qquad\qquad\quad+\Big(\frac{\eps}{2}\sig^2(x)(g_b(x)+g_{b+\del}(x))-x\mu(x)\Big)(g_{b+\del}(x)-g_b(x))\Big]dx.
\end{align*}

Without loss of generality assume $\del \in (0,1)$. By Lemma \ref{lemma:comparison}, $\min_{x\in[y,b]} g_b(x) \le g_b \leq g_{b+\del} \leq g_{b+1}\le C_0$ on $[y,b]$, for some $C_0$, independent of $\delta$.  Since $\sig$ and $\mu$ are bounded on $[y,b]$, Gronwall's inequality implies that there is a constant $C_1 >0$, such that for any $\del \in (0,1)$ we have $$|g_{b+\del}(y) - g_b(y)| \leq C_1|g_{b+\del}(b) - g_b(b)|.$$ The following sequence of relations show that the right-hand side is of order $O(\del)$. We have
\begin{align*}
| g_b(b)-g_{b+\del}(b)| &=|g_{b+\del}(b+\del) - g_{b+\del}(b)|
\\
 &\leq \del\cdot\sup_{x \in [b, b+\del]}|g_{b+\del}'(x)|
 \\
 &\leq \del\cdot\sup_{x \in [b, b+\del]}\Big|\frac{2x\mu(x)}{\sig^2(x)}+2\eps+\frac{\eps \Leps(b+\delta)}{\sig^2(x)}\Big|\cdot\sup_{x \in [b, b+\del]}|g_{b+\del}(x)|
 \\
 &\leq C_2\del,
\end{align*}
where $C_2 > 0$ is independent of $\del$. The equality follows since $g_{b+\del}(b+\del) = g_b(b) = 1$. The first inequality follows by the mean-value theorem. The second inequality follows by the ODE that $g_{b+\del}$ satisfies. Finally, the last inequality follows since all the terms involved are bounded in $[b,b+\del]$, uniformly in $\del$, where for the last term, we used that $1\le g_{b+\del} \leq g_{b+1}$, where the first inequality follows by Lemma \ref{lemma:geqone} and $b + \del > b\geq x^\eps$.
\end{proof}

Set $v^\eps$ as follows:
\begin{align}\label{eq:v_eps}
v^\eps(x)=
\begin{cases}
\int_{\beta^\eps}^xg_{\beta^\eps}(y)dy,& x\in(0,\beta^\eps],\\
x-g_{\beta^\eps}(\beta^\eps),&x\in(\beta^\eps,\iy).
\end{cases}
\end{align}
We are now ready to prove Proposition \ref{prop:HJB}.
\begin{proof}[Proof of Proposition \ref{prop:HJB}]\label{proof:prop:HJB}
Lemmas \ref{lemma:brange} and \ref{lemma:continuity} give that $x^\eps \leq \beta^\eps \leq \bar x^\eps$, and $\beta^\eps \in B^\eps$. By setting $(v^\eps)' = g_{\beta^\eps}$ on $(0,\beta^\eps]$, we get that the triplet $(\beta^\eps,\ell^\eps, v^\eps)$ satisfies the first line of \eqref{HJB_cutoff}, by definition of $g_{\beta^\eps}$ and $\beta^\eps$. For the second line of \eqref{HJB_cutoff}, we set $(v^\eps)' = 1$ on $(\beta^\eps,\iy)$, then by Assumption \ref{assumptions:main}, $\Leps$ is decreasing in $[x^\eps, \iy)$, so $\calL^\eps v^\eps(x) = \Leps(x) \leq \Leps(\beta^\eps) = \ell^\eps$, the second line is satisfied as well. The optimality of $\ell^\eps$ follows by the definition of $\beta^\eps$ as the infimum of $B^\eps$ and the fact that $\la^\eps$ decreases on $[x^\eps,\bar x^\eps]$.

\end{proof}

\subsection{Proof of Proposition \ref{prop:bounded}.}\label{proof:prop:bounded}
Notice that since $\beta^\eps \in  B^\eps$ (see the proof of Proposition \ref{prop:HJB}), then by Lemma \ref{lemma:brange}, strict inequality holds, that is, $\beta^\eps > x^\eps$. Now we aim to prove Proposition \ref{prop:bounded}, where we need to bound $\sig(x)g_b(x)$. The next lemma bounds $g_b(x)$ near 0 for $b < \beta^\eps$ first.

\begin{lemma}\label{lemma:leqone}
For any $ b \in (x^\eps, \beta^\eps)$, there exists $y_1 \in (0,x^\eps)$, such that $g_b(x) \leq 1$ for $x \in (0,y_1]$.
\end{lemma}

\begin{proof}[Proof of Lemma \ref{lemma:leqone}]
Fix $b \in (x^\eps, \beta^\eps)$. Since we take $b < \beta^\eps$, the definition of $\beta^\eps$ implies the existence of $y_1$ with $g_b(y_1) < 1$, and by Lemma \ref{lemma:geqone}, $y_1 < y_2$, where $\Leps(y_2) =\Leps(b)$. We show that for any $x \in (0, y_1)$, $g_b(x)<1$. Arguing by contradiction, assume it does not hold. Let $x_4: = \sup\{x \in (0, y_1): g_b(x) = 1\}$. Then, $$\frac 12\sigma^2(x_4) g_b'(x_4)=\Leps(b) - \Leps(x_4) > 0,$$ where inequality follows since $\Leps(b) > \Leps(x_4)$, because $x_4 < y_1 <y_2$, and by Assumption \ref{assumptions:main} (A2). This implies $g_b'(x_4) > 0$, which contradicts Lemma \ref{lemma:deriv}.
\end{proof}


In order to prove Proposition \ref{prop:bounded}, we need to analyze $h_b:= \sig g_b$. We do it for boundary points $b$'s in a left neighborhood of $\beta^\eps$, establishing a bound and an ODE for $h_b$. Eventually, the continuity for the boundary $b=\beta^\eps$ will be drived using the previous lemma \ref{lemma:continuity}.


\begin{lemma}\label{lemma:unibou}
For any $ b \in (x^\eps, \beta^\eps)$, $h_b:=\sigma g_b$ satisfies:
\begin{align*}
\begin{cases}
\frac{1}{2}\sigma(x)h_b'(x) -\frac{1}{2}\sigma'(x)h_b(x) + \frac{x\mu(x)}{\sigma(x)}h_b(x) - \frac{\eps}{2}h_b^2(x) = \Leps(b),\qquad x\in(0,b],
\\h_b(b) = \sigma(b).
\end{cases}
\end{align*}
Moreover, $h_{b}(x) \leq \sigma(\beta^\eps)$ for $x \in (0,b)$.
\end{lemma}

\begin{proof}[Proof of Lemma \ref{lemma:unibou}]
Fix $ b \in (x^\eps, \beta^\eps)$. From the ODE for $g_b$, it follows that $h_b$ satisfies the ODE above. Plugging in $x=b$ in the ODE above and using the boundary condition $h_b(b) = \sigma(b)$, we get by Assumption \ref{assumptions:main} (A1) that $\frac{1}{2}\sigma(b)h_b'(b) = \frac{1}{2}\sigma'(b)h_b(b) =\frac{1}{2}\sigma'(b)\sigma(b)> 0$. As a result, there exists $\delta_b > 0$ such that for $x \in [b-\delta_b, b)$, $h_b(x) \leq h_b(b) = \sigma(b)$. By Lemma \ref{lemma:leqone}, there is a point $y_1 \in (0,b)$ such that for $x \in (0, y_1]$, $h_b(x) \leq \sigma(x) \leq \sigma(b)$. We now show that for any $x \in (y_1, b-\delta_b)$ we have
$$h_b(x) \leq \sigma(b).$$
Arguing by contradiction, suppose it does not hold. Then there exists $y_3 \in (y_1, b-\delta_b)$ such that $h_b(y_3) > \sigma(b)$. Let $y_4 := \inf\{x \in (y_3, b - \delta_b): g_b(x) = 1\}$ and $y_5 := \sup\{x \in (y_1, y_3): g_b(x) = 1\} $ be the first times to the right and left of $y_3$, where $h_b$ attains the value $\sigma(b)$. We have $y_5 < y_3 < y_4$, and by Lemma \ref{lemma:deriv} we have $\frac{1}{2}\sigma(y_4)h_b'(y_4) \leq 0 \leq \frac{1}{2}\sigma(y_5)h_b'(y_5)$. However, since $x\mu(x)/\sigma(x)$ is decreasing, we have $\frac{y_5\mu(y_5)}{\sigma(y_5)}\sigma(b) > \frac{y_4\mu(y_4)}{\sigma(y_4)}\sigma(b)$, and since $\sigma'$ is nondecreasing, we have
\begin{align}\notag
\frac{1}{2}\sigma(y_5)h_b'(y_5) &= \la^\eps(b) + \frac{1}{2}\sigma'(y_5)\sigma(b)+\frac{\eps}{2}\sigma^2(b)-\frac{y_5\mu(y_5)}{\sigma(y_5)}\sigma(b) \\\notag
&< \la^\eps(b) + \frac{1}{2}\sigma'(y_4)\sigma(b)+\frac{\eps}{2}\sigma^2(b)-\frac{y_4\mu(y_4)}{\sigma(y_4)}\sigma(b) = \frac{1}{2}\sigma(y_4)h_b'(y_4),
\end{align}
a contradiction.

We have shown that for all $x \in (0,b)$ one has $h_b(x) \leq \sigma(b)$. Since $b < \beta^\eps$ and $\sigma(x)$ is increasing the proof is complete.
\end{proof}

\begin{proof}[Proof of Proposition \ref{prop:bounded}]
Fix any point $x \in (0, \beta^\eps)$ and take an increasing sequence $\{b_i\}_i \subseteq (x,\beta^\eps)$ that converges to $\beta^\eps$. From Lemma \ref{lemma:unibou} we have that for each $b_i$ it is true that $h_{b_i}(x) \leq \sigma(\beta^\eps)$. Then, by Lemma \ref{lemma:continuity} as $b_i\uparrow \beta^\eps$ we have that $h_{b_i}(x) = \sigma(x)g_{b_i}(x)$ converges to $\sigma(x)g_{\beta^\eps}(x)$. Hence, $\sigma(x)g_{\beta^\eps}(x) \leq \sigma(\beta^\eps)$ and $\sigma(\beta^\eps)g_{\beta^\eps}(\beta^\eps) = \sigma(\beta^\eps)$. Combining the above and recalling the definition of $v^\eps$ from \eqref{eq:v_eps}, we get that that for any $x\in(0,\beta^\eps]$ one has $\sigma(x)(v^\eps)'(x) \leq \sigma(\beta^\eps)$.
\end{proof}

\subsection{Proof of Proposition \ref{prop:upper}}\label{sec:5}
In this part we fix the gap from \cite{alv-hen2019}. The crux of the matter is that the function $v^\eps$ may be unbounded from below, hence $\E^{\Q^\eps}[v^\eps(X^Z_T)]$ may be equal to $-\iy$. To overcome this challenge, we work with a truncated version of a function that is associated with a threshold that is arbitrarily close to $\beta^\eps$ from below.

Fix $T ,x> 0$, $\eps\ge 0$, and $b\in[x^\eps,\beta^\eps)$. Recall the definition of the function $g_b$ from \eqref{eq:gbgamma}. Set
$\al_b := \inf\{x>0: g_b(y) \geq 1, y \in [x,b]\}$. From Lemmas \ref{lemma:passboundary} and \ref{lemma:leqone} it follows that $\alpha_b\in(0,b)$ and moreover that
$g_b(x) <1$ (resp., $g_b(x)\ge 1$) for $x \in(0,\al_b)$ (resp., $x\in[\al_b,\iy)$). Note that $(v^\eps)'(x)=g_{\beta^\eps}(x) \geq 1$ for all $x \in (0,\beta^\eps]$ (in fact, it can not equal to 1 all the way, which makes it a constant function, and then $\calL^\eps v_{\beta^\eps}(x) = \lambda^\eps(x)$, for all $x$, contradicting \eqref{HJB_cutoff}). Thus, from Lemma \ref{lemma:continuity}, we have the convergence
\begin{align}\notag
\al_b \to 0+\qquad\text{as}\qquad b \to \beta^\eps-.
\end{align}

Define the function $ v_b$ via $v_b(x)=v_b(\alpha_b)+\int_{\alpha_b}^xv'_b(y)dy$, $x\in(0,\iy)$, where,
\begin{align*}
  v'_b(x) :=
 \begin{cases}
    g_b(x), & x \in[\al_b,\iy),\\
 g_b'(\al_b)x - g_b'(\al_b)\al_b + 1, &x\in(0,\alpha_b).
    \end{cases}
\end{align*}
As mentioned above, $g_b(x)\ge 1$ for $x\in[\alpha_b,\iy)$. Together with the construction we have $v_b'(x)\ge 1$ for any $x\in(0,\iy)$. Another key relation that we show in the sequel is that $v_b$ satisfies  $\calL^\eps v_b(x) \le\ell^\eps+\delta_b$, for some $\delta_b\to 0$ as $b\to \beta^\eps-$. To make the proof more fluent, we assume for now that it holds.

Set an arbitrary admissible control $Z$ and set the stopping times $T_n := T \wedge \inf\{t\geq X^Z_t \notin[1/n,n] \}$. 
Recall the structure of the operator $\calL^\eps$ from \eqref{operator} and define the measure $\Q^{v_b}$ with the Girsanov kernel $\psi^{v_b}$, given by
\begin{align}\notag
\begin{split}
\psi^{v_b}_t&:=\argmin_{p\in\R}\Big\{
\frac 12\sigma^2(X^Z_t)v_b''(X^Z_t)+\Big(X^Z_t\mu(X^Z_t)+\sigma(X^Z_t)p\Big)v_b'(X^Z_t)+\frac 1{2\eps} p^2
\Big\}\\
&\;=-\eps\sigma(X^Z_t) v_b'(X^Z_t).
\end{split}
\end{align}
Note that for any $b>0$, the drift term under the measure $\Q^{v_b}$ is $x\mu(x)-\eps\sigma^2(x)v_b'(x)$, which behaves near $x=0+$ as $x\mu(x)-O(x^2)$ (see \eqref{dynamics:Q}). Therefore, by setting up $S_{\Q^{v_b}}$ as in \eqref{scale} with the drift and variance from \eqref{dynamics:Q}, one gets that (A0) holds under the measure $\Q^{v_b}$ as well, which means that the process $X$ does not get absorbed at $0$ in a finite time.

By It\^o's Lemma, and the definition of $\psi^{v_b}_t$,we get that
\begin{align*}
v_b(X^Z_{T_n}) &= v_b(x) + \int_0^{T_n}\mathcal L^\eps v_b(X^Z_s)ds-\int_0^{T_n}\frac{\eps}{2}\sigma^2(X^Z_s)(v_b'(X^Z_s))^2ds
+\int_0^{T_n}\sigma(X^Z_s)v_b'(X^Z_s)dW^{\Q^{v_b}}_s\\
&\quad-\int_0^{T_n}v_b'(X^Z_{s-})dZ_s+\sum_{s\leq T_n}(v_b'(X^Z_{s-})\Delta Z_s+v_b(X^Z_{s})-v_b(X^Z_{s-})).
\end{align*}
Since $v_b \in \mathcal C^2$, its derivative is bounded on $[1/n,n]$. This, together with the bounded variation of $Z$ gives us,
\begin{align*}
v_b(X^Z_{T_n}) &= v_b(x) + \int_0^{T_n}\mathcal L^\eps v_b(X^Z_s)ds
-\int_0^{T_n}\frac{\eps}{2}\sigma^2(X^Z_s)(v_b'(X^Z_s))^2ds+\int_0^{T_n}\sigma(X^Z_s)v_b'(X^Z_s)dW^{\Q^{v_b}}_s
\\
&\quad-\int_0^{T_n}v_b'(X^Z_s)dZ^c_s+\sum_{s\leq T_n}(v_b(X^Z_{s})-v_b(X^Z_{s-})),
\end{align*}
where $Z^c$ is the continuous part of $Z$ and in case that $\eps=0$, $\Q^{v_b}=\PP$. Rearranging the equation and taking expectation with respect to $\Q^{v_b}$, one gets,
\begin{align*}
&\mathbb E^{\Q^{v_b}} \Big[\int_0^{T_n}dZ_s  +  \int_0^{T_n}\frac{1}{2\eps}(\psi^{v_b}(X^Z_s))^2ds\Big]\\
&\quad= \mathbb E^{\Q^{v_b}} \Big[v_b(x) - v_b(X^Z_{T_n}) -\int_0^{T_n}[v_b'(X^Z_s)-1]dZ^c_s-\sum_{s\le T_n}\int_{X^Z_s}^{X^Z_{s-}}[v_b'(y)-1]dy+\int_0^{T_n
}\mathcal L^\eps v_b(X^Z_s)ds\Big].
\\
&\quad\le \mathbb E^{\Q^{v_b}} \Big[v_b(x) - v_b(X^Z_{T_n})\one_{\{X^Z_{T_n}\le \beta^\eps\}} -v_b(\beta^\eps)\one_{\{X^Z_{T_n}> \beta^\eps\}}\Big]+(\ell^\eps+\delta_b) T_n.
\end{align*}
The equality follows by the identity $(\psi^{v_b}(x))^2/(2\eps) = (\eps/2)\sigma^2(x)(v_b'(x))^2$; the inequality follows since $v_b'(x) \geq 1$ and by our assumption (to be proved below) that $\calL^\eps v_b(x) \le \ell^\eps+\delta_b$. Taking $n\to\iy$, then by 
the monotone convergence theorem, 
\begin{align*}
&\E^{\Q^{v_b}} \Big[\int_0^{T}dZ_s  +  \int_0^{T}\frac1{2\eps}(\psi^{v_b}(X^Z_s))^2ds\Big]\leq v_b(x) +\sup_{y\in(0,\beta^\eps]}|v_b(y)|+( \ell^\eps+\delta_b) T.
\end{align*}
Recall that given $b$, the function $v_b$ is bounded on $(0,\beta^\eps]$.
Dividing both sides by $T$ and taking $\liminf_{T\to\iy}$, one gets
\begin{align*}
\inf_{\Q\in\calQ(x)}J(x,Z,\Q)\le J(x,Z,\Q^{v_b})\le\ell^\eps+\delta_b.
\end{align*}
Finally, recall that $Z$ was an arbitrary control, so by taking supremum over $Z$ on both sides and then $b\to\beta^\eps-$, we get the result.

The rest of the proof is dedicated to showing that $\calL^\eps v_b(x)\le  \ell^\eps+\delta_b$ for $x\in(0,\iy)$. By the construction of $v_b$, $\calL^\eps v_b(x)=\ell^\eps$ for any $x\in[\al_b,\iy)$. Hence, it is only left to show that
$\calL^\eps v_b(x)\le  \ell^\eps+\delta_b$ for any $x\in(0,\alpha_b)$. From the definition of $g_b$, we have at $\alpha_b$:
\begin{align*}
    \frac 12 \sig^2(\al_b) g_b'(\al_b) + \al_b\mu(\al_b) - \frac \eps2 \sig^2(\al_b) = \lambda^\eps(b).
\end{align*}
Then,
\begin{align}\label{gprimeb}
g_b'(\alpha_b)=2(\lambda^\eps(b)-\lambda^\eps(\alpha_b))/\sigma^2(\alpha_b).
\end{align}

Let us write
\begin{align}\label{calL}
    \calL^\eps v_b(x) &=L^1(x)+L^2(x)+L^3(x)+L^4(x),
    \end{align}
    where
 \begin{align*}
 L^1(x)&:=\frac{1}{2}g_b'(\alpha_b)\bar \sig^2 x^2 + \bar\mu x(g_b'(\alpha_b)(x-\al_b) + 1) - \frac \eps2 \bar \sig^2 x^2(g_b'(\alpha_b)(x-\al_b) + 1)^2,\\
L^2(x)&:=\frac 12 (\sig^2(x) - \bar \sig^2 x^2)g_b'(\alpha_b), \\
L^3(x)&:=
 (x\mu(x) - x \bar\mu)(g_b'(\alpha_b)(x-\al_b) + 1) ,\\
L^4(x)&:=\frac \eps2 (\sig^2(x) - \bar \sig^2 x^2)(g_b'(\alpha_b)(x-\al_b) + 1)^2.
\end{align*}
The function $L^1(x)$ is a 4th order polynomial in $x$. Differentiate it:
\begin{align*}
    (L^1)'(x) &:= g_b'(\alpha_b)\bar \sig^2 x + \bar\mu(g_b'(\alpha_b)(x-\al_b) + 1)+\bar\mu xg_b'(\alpha_b) \\
    &\quad- \eps\bar \sig^2x(g_b'(\alpha_b)(x-\al_b) + 1)^2-\eps \bar \sig^2x^2 g_b'(\alpha_b)(g_b'(\alpha_b)(x-\al_b) + 1).
\end{align*}
By our assumptions on the parameters, this is a cubic polynomial with negative leading coefficient (consider its domain to be $\R$), and hence goes to $-\iy$ (resp., $+\iy$) for sufficient large (resp., small) $x$.
Now plug in $\al_b$ and $0$ to get
\begin{align*}
(L^1)'(\al_b)& = g_b'(\alpha_b)\bar \sig^2 \al_b + \bar\mu + \bar\mu g_b'(\alpha_b)\al_b - \eps\bar \sig^2\al_b - \eps \bar \sig^2g_b'(\alpha_b)\al_b^2,\\
   (L^1)'(0) &= \bar\mu(1-g_b'(\alpha_b)\al_b).
\end{align*}
By the quadratic bound for $\sigma$ in Assumption \ref{assumptions:main} (A1), there exist $c_1,c_2>0$, independent of $b$, such that $c_1\al_b^2 \leq \sig^2(\al_b) \leq c_2\al_b^2$. Moreover, $b\mapsto \la^\eps(b)-\la^\eps(\alpha_b)$ is bounded. Thus, from \eqref{gprimeb}, $g_b'(\al_b)\al_b \to \iy$ and $g_b'(\al_b)\al_b^2$ is bounded as $\al_b \to 0+$. Therefore, $(L^1)'(\al_b) > 0$ and $(L^1)'(0) < 0$ for sufficiently small $\al_b$. By some basic knowledge about the shape of cubic functions, this indicates that the function $(L^1)'$ has three zeros, one in each of the intervals: $(-\iy,0), (0,\alpha_b), (\alpha_b,\iy)$.
This means that on the interval $[0,\alpha_b]$, the function $L^1$ first decreases and then increases. That is, it obtains its maximum on this interval at one of the endpoints: $0,\alpha_b$. By substitution, we get that for any $x\in[0,\alpha_b]$, $$L^1(x)\le L^1(\alpha_b)=\la^\eps(b)-L^2(\alpha_b)-L^3(\alpha_b)-L^4(\alpha_b).$$

Again, the quadratic bounds in Assumption \ref{assumptions:main} (A1) imply that $|\sig^2(x) - \bar \sig^2 x^2| \leq 2cx^2\max\{\sig(x), \bar \sig x\}$ and for any  $x\in(0, \al_b]$,
\begin{align*}
&|L^2(x)| \leq c\al_b^2\max\{\sig(\al_b), \bar \sig \al_b\}g_b'(\alpha_b),\\
&|L^3(x)| \leq c\al_b^2\max\{1,|1-g_b'(\alpha_b)\al_b|\},\\
&|L^4(x)| \leq \eps c\al_b^2\max\{\sig(\al_b), \bar \sig \al_b\}\max\{1,|1-g_b'(\alpha_b)\al_b|\}.
\end{align*}
Since $\al_b^2/\sig^2(\al_b)$ is bounded as $\al_b \to 0+$, and $\max\{\sig(\al_b), \bar \sig \al_b\} \to 0$ as $\al_b \to 0+$, the three right-hand sides converge to 0. Finally, use the last bounds together with the continuity of $\lambda^\eps$ (at $\beta^\eps$) and \eqref{calL} to get that
$\calL^\eps v_b(x)\le \ell^\eps+\delta_b$, where $\delta_b\to0$ as $b\to\beta^\eps-$.

\qed
\section{Comparative statics}\label{sec:6}
In this section we analyze the monotonicity of the parameters $\ell^\eps$ and $\beta^\eps$ with respect to $\eps$ and their limiting behavior as $\eps\to0+$ and $\eps\to\iy$. In this way, we show the convergence of our model to the risk-neutral model studied by \cite{alv-hen2019}. Recall that at least for the Verhulst--Pearl model given in Example \ref{ex:VP}, Assumptions (A0)--(A2) hold for any $\eps\ge 0$, and therefore, in this example Theorem \ref{thm:main} is valid on the entire region $[0,\iy)$ for $\eps$. 
%

\begin{theorem}\label{thm:cont}
The mapping $[0,\infty) \ni \eps \mapsto \beta^\eps$ is non-increasing and and $\lim_{\eps \to 0+}\beta^\eps = \beta^0$ and $\lim_{\eps \to \infty}\beta^\eps = 0$.
\end{theorem}
\begin{proof}[Proof of Theorem \ref{thm:cont}]

We first prove monotonicity. Fix $0 \leq \eps_2 < \eps_1$. For every $\gam > 0$, let $f_{1,\gam}$ be the solution of the following ODE:
\begin{align*}
\begin{cases}
\frac 12\sigma^2(x) f_{1,\gam}'(x)+x\mu(x) f_{1,\gam}(x)-\frac{\eps_1}{2}\sigma^2(x)f_{1,\gam}^2(x)=\la^{\eps_1}(\beta^{\eps_2})-\gamma,\qquad x\in(0,\beta^{\eps_2}],
\\
f_{1,\gam}(\beta^{\eps_2})=1.
\end{cases}
\end{align*}
Set $f_2 = (v^{\eps_2})'=g_{\beta^\eps}$ (see \eqref{eq:v_eps}) 
\begin{align*}
\begin{cases}
\frac 12\sigma^2(x) \varphi '(x)+x\mu(x) \varphi (x)+\frac{\eps_2 - \eps_1}{2}\sigma^2(x)f_2^2(x) \\
\qquad\qquad- \eps_1\sig^2(x)f_2(x)\varphi (x) - \frac{\eps_1}{2}\sig^2(x)\varphi ^2(x)=\la^{\eps_1}(\beta^{\eps_2}) - \ell^{\eps_2}-\gamma,\qquad x\in(0,\beta^{\eps_2}],
\\
\varphi (\beta^{\eps_2})=0.
\end{cases}
\end{align*}
The identities $f_2(\beta^{\eps_2}) = 1$, $\varphi(\beta^{\eps_2}) = 0$, and $\la^{\eps_1}(\beta^{\eps_2}) - \ell^{\eps_2} = \frac{\eps_2 - \eps_1}{2}\sigma^2(\beta^{\eps_2})f_2^2(\beta^{\eps_2})$ give that $\frac 12 \sig^2(\beta^{\eps_2})\varphi'(\beta^{\eps_2}) = -\gam < 0$. We show that for any $x \in (0, \beta^{\eps_2})$ we have $\varphi(x) > 0$. Arguing by contradiction, suppose it does not hold. Set $x_6 := \sup\{x \in (0, \beta^{\eps_2}) : \varphi(x) = 0\}$. Then,
\begin{align*}
   \frac 12 \sig^2(x_6)\varphi'(x_6) &= \la^{\eps_1}(\beta^{\eps_2}) - \ell^{\eps_2} - \frac{\eps_2 - \eps_1}{2}\sigma^2(x_6)f_2^2(x_6) -\gamma \\&= \frac{\eps_2 - \eps_1}{2}(\sigma^2(\beta^{\eps_2}) - \sigma^2(x_6)f_2^2(x_6)) - \gam \\&<0,
\end{align*}
where the inequality follows since by Proposition \ref{prop:bounded}, $\sigma(x_6)f_2(x_6) \leq \sig(\beta^{\eps_2})$, and $\eps_2 < \eps_1$. But this contradicts Lemma \ref{lemma:deriv}. Therefore $\varphi > 0$ on $(0,\beta^{\eps_2})$.

The last conclusion together with $f_2 = (v^{\eps_2})'$ gives $f_{1,\gam}(x) \geq f_2(x) \geq 1$ for $x \in (0,\beta^{\eps_2})$. Since $\gam > 0$ is arbitrary, taking $\gam \to 0+$ and using Lemma \ref{lemma:perturb} in case $\gam = 0$, we have $f_{1,\gam}(x) \geq 1$ for $x \in (0,\beta^{\eps_2})$. This together with the definition of $\beta^{\eps_1}$ implies $\beta^{\eps_1} \leq \beta^{\eps_2}$.

We now turn to prove continuity at $\eps = 0$. First notice the limit of $\beta^\eps$ as $\eps \to 0+$ exists. It simply follows since $\beta^\eps$ is increasing and bounded above by $\beta^0$. We denote the limit by $\hat \beta$. Trivially, we have $\hat \beta \leq \beta^0$, so it is sufficient to show $\hat \beta \geq \beta^0$. For this, let $\hat g$ be the solution to the ODE
\begin{align*}
\begin{cases}
\frac 12\sigma^2(x) \hat g'(x)+x\mu(x) \hat g(x)=\la^{0}(\hat \beta),\qquad x\in(0,\hat \beta],
\\
\hat g(\hat \beta)=1.
\end{cases}
\end{align*}
For any $y \in (0, \hat \beta)$ choose $\eps$ sufficiently close to $0$ such that $\beta^\eps > y$. For any $w \in [y, \beta^\eps]$ we have

\begin{align*}
  &  (v^\eps)'(w) - \hat g(w)
\\&\;= (v^\eps)'(\beta^\eps) - \hat g(\beta^\eps)  - \int_w^{\beta^\eps} \frac{2}{\sig^2(x)}((\la^{0}(\hat \beta) - \ell^\eps) + x\mu(x)((v^\eps)'(x)-\hat g(x))-\frac \eps2 (\sig(x)(v^\eps)'(x))^2)dx.
\end{align*}
Note that $\sig$, $\mu$, and $\sig(v^\eps)'$ are bounded on $[y,\beta^\eps]$, and the bounds can be made independent of $\eps$, because $\beta^\eps \leq \beta^0$. Without loss of generality we can also make $\eps$ bounded by $1$ since $\eps \to 0+$. Gronwall's inequality implies that there is a constant $C_3 > 0$ independent of $\eps$ and $y$, such that $$|(v^\eps)'(y) - \hat g(y)| \leq C_3 |(v^\eps)'(\beta^\eps) - \hat g(\beta^\eps)| = C_3 |1 - \hat g(\beta^\eps)| = C_3 |\hat g(\hat \beta) - \hat g(\beta^\eps)|.$$ This term converges to $0$ since $\beta^\eps \to \hat \beta$. This implies $\hat g(y) \geq 1$, because $(v^\eps)'(y) \geq 1$. Since $y$ is arbitrary, it follows that for all $y \in (0, \hat \beta]$ we have $\hat g(y) \geq 1$. This in turn implies $\hat \beta \geq \beta^0$.

Finally, $0\le \lim_{\eps\to\infty}\beta^\eps\le \lim_{\eps\to\infty}\bar x^\eps=0$.
\end{proof}

\begin{theorem}\label{thm:valuecont}
The mapping $[0, \infty) \ni \eps \mapsto \ell^\eps$ is non-increasing and $\lim_{\eps \to 0+}\ell^\eps = \ell^0$ and $\lim_{\eps \to \infty}\ell^\eps = 0$.
\end{theorem}
\begin{proof}[Proof of Theorem \ref{thm:valuecont}]
Fix $0 < \eps_2 <\eps_1$ and $\delta>0$. For any admissible control $Z\in\calZ(x)$ for $\eps_2$ consider a $\delta$-optimal control $\Q^{Z,\eps_2}\in\calQ(x)$, so that $\sup_{Z \in \calZ(x)} J^{\eps_2}(x, Z, \Q^{Z,\eps_2})\le \ell^{\eps_2}+\delta$.
Let $\psi^{Z,\eps_2}$ be its Girsanov's kernel. Then,
%
\begin{align*}
\ell^{\eps_1} &= \sup_{Z \in \calZ(x)} \inf_{\Q \in \calQ(x)} J^{\eps_1}(x, Z, \Q) \\&\leq \sup_{Z \in \calZ(x)} J^{\eps_1}(x, Z, \Q^{Z,\eps_2}) \\&= \sup_{Z \in \calZ(x)} \Big[J^{\eps_2}(x, Z, \Q^{Z,\eps_2}) + \frac 12 \Big(\frac {1}{\eps_1} - \frac {1}{\eps_2}\Big)\liminf_{T \to \infty} \frac 1T \E^{\Q^{\eps_2}}\Big[\int_0^T (\psi^{Z,\eps_2}_t)^2dt\Big]\Big] \\&< \sup_{Z \in \calZ(x)} J^{\eps_2}(x, Z, \Q^{Z,\eps_2}) \\&= \ell^{\eps_2}+\delta.
\end{align*}
%
Sending $\delta\to0+$, this establishes the monotonicity 
of $\eps \mapsto \ell^\eps$ for $\eps > 0$. The monotonicity at $\eps = 0$ follows by
\begin{align*}
    \ell^{\eps_1} &= \sup_{Z \in \calZ(x)} \inf_{\Q \in \calQ(x)} J^{\eps_1}(x, Z, \Q) \\&\leq \sup_{Z \in \calZ(x)} J^{\eps_1}(x, Z, \PP) \\&= \sup_{Z \in \calZ(x)} \liminf_{T \to \infty} \frac 1T \E^{\PP}\Big[\int_0^T dZ_t\Big] \\&= \ell^0.
\end{align*}
We now turn to proving the continuity at $\eps = 0$. We have
\begin{align*}
    |\ell^\eps - \ell^0| &= |\Leps(\beta^\eps) - L^0(\beta^0)| \\
    &\leq |\beta^\eps \mu(\beta^\eps) - \beta^0 \mu(\beta^0)| + \frac \eps2|\sig^2(\beta^\eps)| \\
    &\leq |\beta^\eps \mu(\beta^\eps) - \beta^0 \mu(\beta^0)| + \frac \eps2|\sig^2(\beta^0)|,
\end{align*}
where the first inequality follows by the triangle inequality, and the second inequality follows since $\sig$ is increasing and $\beta^\eps \leq \beta^0$. Since $\beta^\eps$ is continuous at $\eps = 0$ and $x \mapsto x\mu(x)$ is continuous, we have that
\[
|\beta^\eps \mu(\beta^\eps) - \beta^0 \mu(\beta^0)| + \frac \eps2|\sig^2(\beta^0)|\to 0\qquad\text{as}\qquad\eps\to 0+.
\]

We now turn to proving that $\lim_{\eps \to \infty}\ell^\eps = 0$. For this, first note that $\bar x^\eps \to 0$ as $\eps \to \infty$ because for any fixed $x > 0$, for sufficiently large $\eps$, $\Leps(x) = x\mu(x) - \frac \eps2 \sig^2(x) < 0$. As a consequence we get that $\beta^\eps$, which is bounded above by $\bar x^\eps$, also converges to $0$ as $\eps \to \infty$. Finally $\ell^\eps = \Leps(\beta^\eps) \leq \beta^\eps\mu(\beta^\eps) \to 0$ as $\eps\to\iy$.
\end{proof}

\footnotesize
\bibliographystyle{abbrv}
\bibliography{refs(revised)}
\end{document}